\documentclass[11pt]{article}
\usepackage{graphicx} % Required for inserting images
\usepackage{amsmath,amssymb,amsthm}
\usepackage{mathtools,url}
\usepackage{graphicx}
\usepackage{comment}
\usepackage[dvipsnames]{xcolor}
\usepackage{letltxmacro}
\makeatletter
\let\oldr@@t\r@@t
\def\r@@t#1#2{%
\setbox0=\hbox{$\oldr@@t#1{#2\,}$}\dimen0=\ht0
\advance\dimen0-0.2\ht0
\setbox2=\hbox{\vrule height\ht0 depth -\dimen0}%
{\box0\lower0.4pt\box2}}
\LetLtxMacro{\oldsqrt}{\sqrt}
\renewcommand*{\sqrt}[2][\ ]{\oldsqrt[#1]{#2}}
\makeatother

\makeatletter
\DeclarePairedDelimiterX{\pmodx}[1]{(}{)}{{\operator@font mod}\mkern6mu#1}
\renewcommand{\pmod}{%
  \allowbreak
  \if@display\mkern18mu\else\mkern8mu\fi
  \pmodx
}
\makeatother

\theoremstyle{plain}
\newtheorem{theorem}{Theorem}
\newtheorem*{theorem*}{Theorem}
\newtheorem{lemma}{Lemma}
\newtheorem{corollary}{Corollary}
\newtheorem*{corollary*}{Corollary}
\newtheorem{definition}{Definition}

\theoremstyle{definition}

\newtheorem*{remark*}{Remark}

\newcommand{\+}[1]{{#1}^\times}
  
\def\Aff{\operatorname{Aff}}

\DeclareMathOperator{\ord}{ord}

\def\R{{\mathbb R}}

\def\Z{{\mathbb Z}}
\def\C{{\mathbb C}}

\def\H{\widetilde{H}}

\def\a{{\alpha}}

\def\s{{\sigma}}
\def\k{{\varkappa}}
\def\pfi{{\varphi}}
\def\e{{\varepsilon}}

\def\id{{\rm id\,}}

\title{Panmagic permutations and $N$-ary groups}
\author{Sergiy Koshkin$^*$ and Jaeho Lee$^\dagger$\\
\\
$^*$Corresponding author\\
Department of Mathematics and Statistics\\
University of Houston-Downtown\\
1 Main Street\\
Houston, TX 77002\\
e-mail: \texttt{koshkins@uhd.edu}
\\
\\
$^\dagger$Spring Branch Academic Institute\\
14400 Fern Drive\\
Houston, TX 77079\\
e-mail: \texttt{leejaeho0802@gmail.com} 
}

\date{}

\begin{document}

\maketitle

\begin{abstract}
Panmagic permutations are permutations whose matrices are panmagic squares. Positions of $1$-s in the latter describe maximal configurations of non-attacking queens on a toroidal chessboard. Some of them, affine panmagic permutations, can be conveniently described by linear formulas of modular arithmetic, and we show that their sets have remarkable algebraic properties when one multiplies three or more of them rather than just two. In group-theoretic terms, they are special cosets of the dihedral group in the group of all affine permutations. We also investigate decomposition of panmagic permutations into disjoint cycles and find many connections with classical topics of number theory: multiplicative orders, $4k+1$ primes, primitive roots and quadratic residues.
\medskip

\medskip

\end{abstract}

\section{Introduction}

Magic squares are square matrices with the same sum, called the {\it magic sum}, in each row, each column, the main diagonal and the main anti-diagonal. Their history goes as far back as ancient China c.\,200 BC, and one appears in Albrecht  D\"urer's famous painting {\it Melencolia I} (1514). They attracted the attention of mathematicians since Leonhard Euler's 1776 article linked them to Latin squares \cite{BS2}. Magic squares that additionally have the same magic sum over all `broken' diagonals and anti-diagonals were once called ``diabolic" \cite{Lehm}, but now are commonly known as pandiagonally magic or {\it panmagic} for short \cite{AlKi,BS1}. Early work mainly focused on constructions of magic and panmagic squares, but more recently their algebraic properties also received attention. 

Back in 1950, Derek Lawden claimed that in even dimensions multiplying any three panmagic squares (and then any odd number of them) produces another panmagic square \cite{Lawd}. Although this only holds for panmagic squares with additional symmetry, Lawden's was the first result (that we know of) on closure under ternary multiplication for a class of magic squares. It went entirely unnoticed, but other such classes were discovered in 1990s. Anthony Thompson might have been the first to prove it in print for $3\times3$ magic and $5\times5$ panmagic squares in 1994 \cite{Thom}. In 2012, known classes of magic and panmagic squares closed under ternary multiplication were codified and generalized to higher dimensions by Ronald Nordgren \cite{Nord}. 

Thompson even found a spanning set of $5\times5$ panmagic permutation matrices that is itself closed under ternary multiplication, no such spanning set exists for $3\times3$ magic or Lawden's panmagic squares. Moreover, this spanning set is closely associated with the well-known dihedral group $D_5$, the group of symmetries of the regular pentagon. The nature of this association remained mysterious, like magic, but it suggests looking deeper into permutations with panmagic matrices, which are simpler objects than the matrices themselves.

Although quite natural, the term ``panmagic permutations" is rarely used (it is used in \cite{AlKi}), but they are classically known under a different name. A problem from the mid-19th century, often misattributed to Carl Friedrich Gauss \cite{BS2}, asked to place $8$ non-attacking queens on the $8\times 8$ chessboard, and it was generalized to $n\times n$ chessboards by Fran\c cois Lionnet in 1869. In 1900, George Carpenter proposed to fold the board into a cylinder by identifying two opposite sides of it \cite{Carp}, and in 1918 George Polya folded it further into a torus. Both modifications have the same effect of letting the queens attack along the entire broken diagonals, and came to be known as the {\it modular $n$-queens problem} \cite{BS2}. Given its solution, replace the board with a matrix and place $1$-s into the positions of the queens and $0$-s elsewhere as on Figure\,\ref{ChessPerm}. The result is a panmagic permutation matrix, and conversely, any panmagic permutation produces a modular $n$-queens solution. 
\begin{figure}[!ht]
\vspace{-0.3em}
\begin{centering}
\includegraphics[scale=0.18]{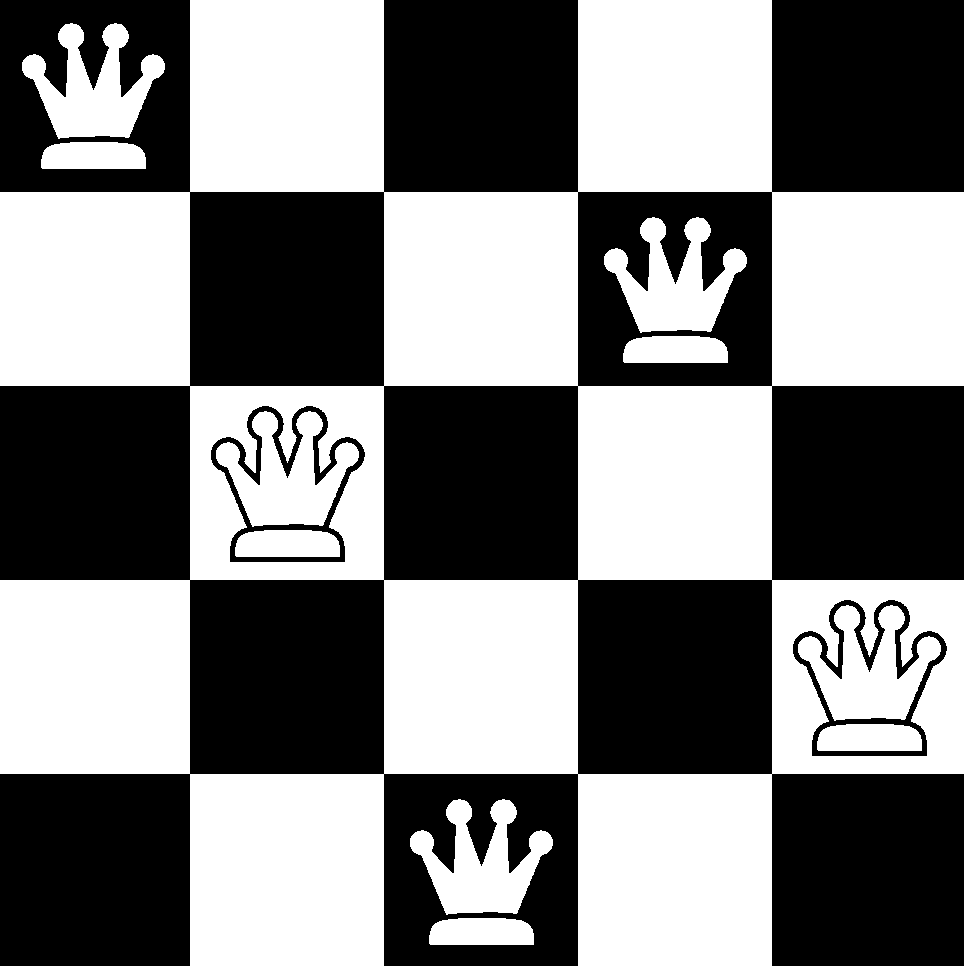}\hspace{3em} \includegraphics[scale=0.9]{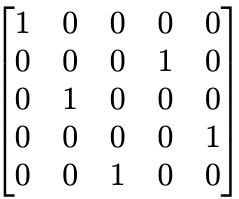}
\par\end{centering}
\caption{\label{ChessPerm} Modular $5$-queens solution and its permutation matrix.}
\end{figure}
\vspace{-0.5em}
In this guise, panmagic permutations have been extensively studied \cite{BS2}, but not, it seems, for their algebraic properties. Even the connection between modular $n$-queens and panmagic squares is little known \cite{BS1} despite being used as early as in 1900 by Charles Planck \cite{BS2}.

For that matter, ternary and more general $N$-ary groups are also little known, but the topic is, again, classical. It goes back to Edward Kasner's talk at the 1911 International Congress of Mathematicians and was originally developed by Wilhelm D\"ornte in 1920-s, at the urging of Emmy Noether \cite{GBV,Shah}. Emil Post, better known for his work in mathematical logic, proved a key structural result about them in 1940, the Post coset theorem. As will be explained in Section \ref{NarySub}, it says, essentially,  that $N$-ary groups are cosets of normal subgroups when the quotient group is cyclic of order $N-1$, and demystifies the appearance of $D_5$ in Thompson's proof.  

We will show that closure under $N$-ary multiplication is a general phenomenon for sets of panmagic permutations in prime dimensions $n=p>3$ because those sets are Post cosets of the dihedral group $D_p$. Even for the ternary case, our descriptions are simpler and much more explicit than those based on linear algebra in \cite{Nord}. A key idea is to consider a subgroup $\Aff(\Z_p)$ of the symmetric group of all permutations $S_p$ in which $D_p$ is normal. Its elements are called {\it affine permutations} and they can be expressed by simple formulas of modular arithmetic in $\Z_p$. We will also describe the cyclic structure of panmagic affine permutations and relate it to some classical topics of number theory: primitive roots, multiplicative orders, $4k+1$ primes and quadratic residues \cite{IrRos,Rosen}. Although we only consider the case of prime dimensions to keep the exposition elementary, we expect most results to generalize to composite $n$.

As many authors pointed out, magic squares provide a concrete and appealing setting for grasping non-trivial concepts of undergraduate linear algebra \cite{Ess,Mat,Thom}. We will show that panmagic permutations are more than apt to do the same for group theory and number theory. Indeed, it is treating them as permutations encoded by simple formulas of modular arithmetic that turns out to be most fruitful for discovering and proving algebraic patterns obscured in other representations.

\section{Permutations and magic squares}

As is standard, we will denote the group of all permutations of $1,2,\dots,n$, the symmetric group, $S_n$, and its identity element $\id$\! \cite{Gallian}. Permutations will often be written in cyclic notation, e.g. $(1)(2\,4\,5\,3)$, where $1$ is a fixed point, $2$ goes to $4$, $4$ to $5$, and so on. We will identify $S_n$ with the group of permutations of $\Z_n=\{0,1,\dots n-1\}$, the set of residues (congruence classes) modulo $n$. To recover a standard permutation of $1,\dots,n$, one simply replaces residue $0$ by $n$. Two special permutations will play a central role in our study, even though neither of them is magic.
\begin{definition}
Denote $\phi$ the flip permutation that swaps antipodal numbers, i.e. $\phi(i):=n+1-i$, and by $\k$ the cyclic shift permutation $\k(i)\equiv i+1\pmod{n}$ that shifts numbers one unit up and $n$ to $1$. In the cyclic notation, $\phi=(1\,n)(2\,n-1)\dots$ and $\k=(1\,2\dots n)$.
\end{definition} 
As usual, $a\mid b$ means that $b$ is an integer multiple of $a$, and $a\equiv b\pmod{n}$ means that $n\mid(b-a)$ \cite{Rosen}. We will often write congruence formulas for residues in a simplified notation that omits $n$ when it is understood, and replaces $\equiv$ by $=$. For example, we may write $n\equiv0\pmod{n}$ as just $n=0$ when no confusion results. In our simplified notation, $\phi(i)=1-i$ and $\k(i)=i+1$. The same convention applies to matrix indices when they go over $n$ or under $1$, for example, in an $n\times n$ matrix $a_{n+1\,1}$ is $a_{11}$. The {\it broken diagonals} of an $n\times n$ matrix $A$ have those $a_{ij}$ where $i-j=\ $constant, and the {\it broken anti-diagonals} have those where $i+j=\ $constant. The {\it main diagonal/anti-diagonal} is the one with the constant equal to $0$.

A square matrix $A$ is called {\it semimagic} when all of its column sums and row sums are the same. Their common value is called the {\it magic sum} \cite{Ess,Nord}. We will be mainly interested in matrices $P_\s$ with $\s\in S_n$ that permute the standard basis vectors $e_i$ as in $P_\s e_i:=e_{\s(i)}$. All of them are semimagic with the magic sum $1$ because a permutation matrix must have a single  entry $1$ in each row and column, and the rest are $0$-s. Moreover, $P_{\s\tau}=P_\s P_\tau$, where on the left we have the composition of permutations and on the right the product of matrices. 
\begin{definition}
A semimagic matrix (square) is called magic when its main diagonal and anti-diagonal sums are also equal to the magic sum. And it is called panmagic when its sums over all broken diagonals and anti-diagonals are equal to the magic sum. A permutation will be called magic or panmagic when its matrix is such.
\end{definition} 

The effect of multiplying a matrix by $P_{\phi}$ is to swap $i$-th and $(n+1-i)$-th rows/columns. This also interchanges (broken) diagonals and anti-diagonals. In turn, multiplication of a matrix by $P_{\k}^i=P_{\k^i}$ on the left/right cyclically shifts its rows/columns by $i$ positions and moves the main diagonal to one of the broken diagonals. 

Note that the number of $1$-s on the diagonal of $P_\s$ is the number of {\it fixed points} of $\s$, those for which $\s(i)=i$. Similarly, since $\phi$ flips points about the middle  the number of $1$-s on the anti-diagonal of $P_\s$ is the number of {\it flipped points} of $\s$, those with $\s(i)=n+1-i$. Thus, a permutation is magic when it has exactly one fixed and exactly one flipped point, and it is panmagic when all of its cyclic shifts $\k^i\s$ are also magic.

\section{Panmagic permutations and dihedral group}

In this section, we will introduce the dihedral group in its permutational guise, relate it to panmagic, and reformulate two classical results about modular $n$-queens solutions in the language of permutations. The subgroup of a group generated by elements of a subset $S$, i.e., the set of all finite products of their positive and negative powers, will be denoted $\langle S\rangle$. When $S=\{a_1,\dots,a_k\}$, we write simply $\langle a_1,\dots,a_k\rangle$ .
\begin{definition}
The subgroup of $S_n$ generated by $\phi$ and $\k$ will be called dihedral and denoted $D_n:=\langle \phi,\k\rangle$. Its elements will be called dihedral permutations. 
\end{definition} 
\noindent As a simple application of our trimmed down congruence notation, let us derive a commutation relation for $\phi$ and $\k$. Note that $\k^{-1}(i)=i-1$, so
\begin{equation}\label{comphikappa}
\phi\k(i)=1-(i+1)=(1-i)-1=\k^{-1}\phi(i).
\end{equation}
It is also easy to see that $\phi^2=\k^{n}=\id$, so the commutation relation implies that all permutations in $D_n$ are of the form  $\k^i$ or $\phi\k^i$ with $i=0,\dots,n-1$. Moreover, those are all distinct, so $|D_{n}|=2n$. Some authors denote $D_{2n}$ what we denote $D_{n}$, by the number of elements in it. 

Our subgroup $D_n$ is isomorphic to the usual dihedral group defined as the group of symmetries of the regular $n$-gon \cite{Gallian}. 
\begin{figure}[!ht]
\vspace{-1.1em}
\begin{centering}
\hspace{1.5em}\includegraphics[scale=0.29]{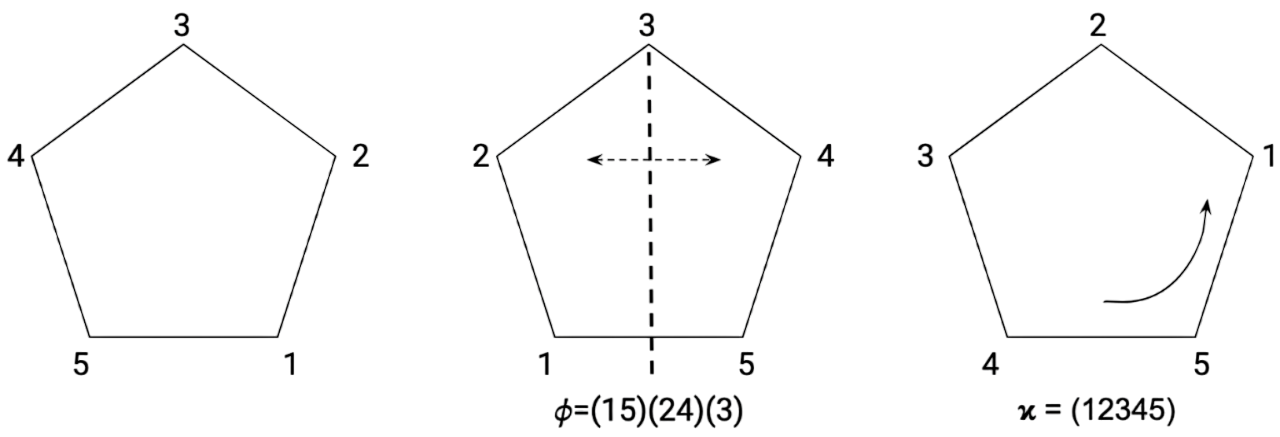}
\end{centering}
\vspace{-1.1em}
\caption{\label{DihGen} Flip $\phi$ as reflection and cyclic shift $\k$ as rotation of the regular pentagon.}
\end{figure}
\vspace{-0.3em}
Indeed, if we number its vertices from $1$ to $n$ then  $\phi$ is the permutation induced by the reflection about the middle perpendicular of one of its sides, and $\k$ is induced by the rotation by the angle $\frac{2\pi}{n}$ around its center, see Figure\,\ref{DihGen}. 

We will now explain how $D_{n}$ is related to panmagic squares and permutations. As multiplication by matrices $P_{\phi}$ and $P_{\k}$ interchanges diagonals and anti-diagonals and shifts rows and columns, respectively, it preserves panmagic. The same is true of multiplication by their products and those are exactly the permutation matrices $P_{\delta}$ with $\delta$ from the dihedral group $D_{n}$. Moreover, multiplication by some $P_{\delta}$ can move any given row/column to any other, and move any given diagonal/anti-diagonal to the main one. Thus, $A$ is a panmagic square if and only if $AP_{\delta}$ (or $P_{\delta}A$) are magic squares with the same magic sum for all $\delta\in D_n$. Accordingly, $\s$ is a panmagic permutation if and only if $\s\delta$ (or $\delta\s$) are magic for all $\delta\in D_n$.

The next theorem is a modular variant of Lionnet's 1869 arithmetic reformulation of the $n$-queens problem \cite{BS1} (see also \cite{AlKi}), and gives convenient arithmetic characterizations of dihedral, magic, and panmagic permutations. 
\begin{theorem}[Lionnet]\label{pmagic criteria} Let $\s\in S_n$ be a permutation. Then
\medskip

\noindent\textup{(i)} $\s$ is dihedral if and only if $\sigma(i)-i$ or $\sigma(i)+i$ is constant;
\medskip

\noindent\textup{(ii)} $\s$ is magic if and only if  $0$ has a unique pre-image modulo $n$ under $\sigma(i)-i$ and $1$ has a unique pre-image modulo $n$ under $\sigma(i)+i$;
\medskip

\noindent\textup{(iii)} 
$\s$ is panmagic if and only if $i\mapsto\sigma(i)\pm i$ are both injective modulo $n$, and hence also permutations. 
\end{theorem}
\begin{proof}\phantom{a} 

\noindent(i) The ``only if" part is trivial from the modular formulas for $\phi$ and $\k$. Conversely, if $\s(i)-i=c$ then $\s(i)=i+c$, so $\s=\k^c$. And if $\s(i)+i=c$ then $\s(i)=1-i+c-1$, so $\s=\k^{c-1}\phi$. In both cases, $\s$ is dihedral.
\smallskip

\noindent(ii) Let $\s$ be magic. Since it has a single fixed point there is only one $i$ with $\s(i)=i$, so $0$ has a unique pre-image under $\s(i)-i$. Since it also has a single flipped point, there is only one $j$ with $\s(j)=1-j$ and $1$ has a unique pre-image under $\s(j)+j$. Running the argument in reverse, we conclude that $\s$ has one fixed and one flipped point, and hence is magic. 
\medskip

\noindent(iii) Let $\s$ be panmagic. Then $\k^{-k}\s$ is magic for any $k$, and $\k^{-k}\s(i)-i=\s(i)-i-k=0$ has a unique solution modulo $n$ by (ii). But then $\s(i)-i=k$ has a unique solution for any $k$ and $\s(i)-i$ is injective. Similarly, $\k^{k}\phi\s$ is magic for any $k$, $\k^{k}\phi\s(i)-i=1-(\s(i)+i)+k=1$ has a unique solution, and then so does $\s(i)+i=k$. Thus, $\s(i)+i$ is also injective. 

Conversely, let $\sigma(i)\pm i$ be injective modulo $n$. By (ii), it is enough to show that so are $\delta\s(i)\pm i$ for all $\delta\in D_n$, and even just for $\delta=\k,\phi$ since they generate $D_n$. But $\k\s(i)\pm i=\sigma(i)\pm i+1$ and $\phi\s(i)\pm i=1-(\sigma(i)\mp i)$, so this follows trivially.
\end{proof}
A sad consequence of Lionnet's characterization is that panmagic permutations do not exist in all dimensions. The `nice' dimensions first appeared in Carpenter's construction of modular $n$-queens solutions in 1900 \cite{Carp}, and were proved exhaustive by Polya in 1918 by tiling the plane with chessboards. The slick congruence proof of the non-trivial ``only if" part below is due to Kl\o ve, see \cite{BS1}. The ``if" part will easily follow from constructions in Section \ref{Affine}.
\begin{theorem}[Polya]\label{PolyaDim} $S_n$ contains panmagic permutations if and only if $2,3\nmid n$. 
\end{theorem}
\begin{proof} Let $\s$ be a panmagic permutation. By Theorem \ref{pmagic criteria}\,(iii), $\s(i)-i$ is also a permutation, which means that for $i=1,\dots,n$ it returns the same residues in a different order. Therefore, modulo $n$: 
$$
\Sigma_1:=\sum_{i=1}^ni=\sum_{i=1}^n(\s(i)-i)=\Sigma_1-\Sigma_1=0,
$$ 
i.e. $\Sigma_1$ is a multiple of $n$. But $\Sigma_1=\frac{n(n+1)}{2}$ and this can only happen when $n+1$ is even, so $2\nmid n$. 

Similarly, let $\Sigma_2:=\sum_{i=1}^ni^2$. Since $\s(i)+i$ is also a permutation, we have modulo $n$: 
$$
2\Sigma_2=\sum_{i=1}^n(\s(i)-i)^2+\sum_{i=1}^n(\s(i)+i)^2= 2\sum_{i=1}^n\s(i)^2+2\sum_{i=1}^ni^2=4\Sigma_2=0,
$$ 
i.e. $2\Sigma_2$ is a multiple of $n$. But $2\Sigma_2=\frac{n(n+1)(2n+1)}{3}$, and this can only happen when $n+1$ or $2n+1$ is divisible by $3$. But both numbers are relatively prime to $n$, so $3\nmid n$. 
\end{proof}
\noindent We will refer to $n>1$ with $2,3\nmid n$ as {\it Polya dimensions}, and $n$ will be assumed to be Polya from now on, unless otherwise stated.

\section{Panmagic products in dimensions $5$ and $7$} \label{57PMag}

To motivate further developments, we will analyze products of panmagic permutations in the two lowest dimensions where they exist, $n=5$ and $n=7$. For the $n=5$ case we will initially follow Thompson \cite{Thom}.

When $H$ is a subgroup of a group $G$ and $a\in G$ we will write $aH$ for the {\it left coset} of $H$, i.e., the set of all products $ah$ with $h\in H$. Thompson presents the space of $5\times5$ panmagic squares as the linear span of permutation matrices with permutations from the coset $\psi D_5$, where 
$\psi:=(1)(2\,4\,5\,3)$ is a particular panmagic permutation. Explicitly,
\begin{equation}\label{D5}
\psi D_5:=\{\psi,\psi\k,\dots,\psi\k^4,\psi\phi,\psi\phi\k,\dots,\psi\phi\k^4\},
\end{equation}
and there is a single independent linear relation among the matrices of $\psi D_5$, namely, $\sum_{i=0}^4P_{\psi\k^i}=\sum_{i=0}^4P_{\psi\phi\k^i}$. The reader can check that both sides add up to the matrix with all entries $1$. The space of $5\times5$ panmagic squares is known to be $9$-dimensional \cite{Hou}, so any $9$ of $\psi D_5$ matrices form a basis of it. By direct calculation, $D_5\cdot\psi D_5=\psi D_5\cdot D_5\subseteq\psi D_5$ and $\psi D_5\cdot\psi D_5\subseteq D_5$, where the product of sets consists of all products of elements in them. It follows immediately that $\psi D_5$ is closed under ternary multiplication.

We can check the above inclusions by using commutation relations. In the congruence notation, $\psi(i)=3-2i$, and it is a simple exercise analogous to \eqref{comphikappa} that $\phi\psi=\psi\phi\k$, $\k\psi=\psi\k^2$ and $\psi^2=\phi\k^{4}$. Therefore, products like $\phi\k^i\cdot\psi\phi\k^j$ are in $\psi D_5$ because we can move $\psi$ to the leftmost position leaving only powers of $\phi$ and $\k$ to the right of it. Similarly, products like $\psi\phi\k^i\cdot\psi\phi\k^j$ are in $D_5$ because we can move the second $\psi$ likewise. At the end, only powers of $\phi$ and $\k$ remain since $\psi^2\in D_5$.

Put this way, the exercise easily extends to $\psi D_7$ written as in \eqref{D5} but with powers of $\k$ up to $6$. We can again take $\psi(i)=3-2i$, i.e., $\psi=(1)(2\,6\,5\,7\,3\,4)$, and the commutation relations for $\phi,\psi$ and $\k,\psi$ remain the same. However, it is no longer true that $\psi D_7\cdot\psi D_7\subseteq D_7$ because $\psi^2\not\in D_7$. This time, $\psi^3\in D_7$, so $\psi^3\not\in\psi D_7$ and $\psi D_7$ is no longer ternary. It is instead {\it quaternary}, as the argument analogous to the one for $D_5$ shows, and $\psi^2D_7$ is also quaternary by the same reasoning. But the union, $\psi D_7\cup\psi^2D_7$, which happens to be the set of all panmagic permutations in $S_7$, is neither ternary, nor quaternary, nor $N$-ary for any $N$. However, if we add $D_7$ itself to this union then the result is an ordinary (binary) group.

Two lessons of $D_7$ generalize to $D_n$ in prime dimensions $n\geq7$. We should expect $N$-arity for some $N$, but not necessarily ternarity, and we should expect it for parts of the panmagic permutation set rather than for the whole of it. The group we obtained by adding the dihedral group to the set of all panmagic permutations, which we will denote $\langle\psi,D_n\rangle$, is the group generated by $\psi$ and $D_n$ jointly. $D_n$ is a normal subgroup in it, i.e., all of its left and and right cosets coincide, $\s D_n=D_n \s$, due to the commutation relations. It is called the Post cover of $\psi D_n$ in $N$-ary group theory \cite{GBV,Shah}. Discerning its identity will be our next task.

\section{Affine permutations and their group}\label{Affine}

The group $\langle\psi,D_n\rangle$, the minimal group that contained all panmagic cosets, cannot be the entire symmetric group $S_n$ of all permutations. Indeed, $D_n$ is never normal in $S_n$ for $n\geq4$ \cite{Gallian}. A clue to its identity is that all permutations in it, $\phi,\k,\psi$ and their products and powers, are expressed by very simple congruence formulas of the form $\s(i)=ai+b$. These represent affine transformations of $\Z_n$ analogous to affine transformations of $\R$ when $a$ and $b$ are real numbers, whose graphs are straight lines with slope $a$.
\begin{definition}\label{pia} A permutation in $S_n$ is called affine when it is given by an invertible affine transformation of $\Z_n$. Affine permutations will be denoted $\pi_{a,b}(i)\equiv ai+b\pmod{n}$ and linear ones $\pi_a:=\pi_{a,0}$, with $a$ called their slope. The group of affine permutations, also known as the general affine group of degree $1$ over $\Z_n$, is denoted $\Aff(\Z_n)$. 
\end{definition}
\noindent Aside from modular $n$-queens, affine permutations come up in pseudo-random number generation \cite{HD,Mars}, in enumerative combinatorics \cite{BK-V}, and in pure number theory, for example, in Zolotarev's lemma. As with the dihedral group $D_n$, we will not distinguish between $\Aff(\Z_n)$ and its isomorphic representation by permutations, i.e. we will also treat it as a subgroup of $S_n$. Clearly, $\pi_{a,b}$ is invertible, and hence a permutation, if and only if $a$ is invertible in $\Z_n$. Let us denote $\+{\Z}_n$ the group of invertible residues of $\Z_n$, called {\it units}, then
\begin{equation}\label{GAmod}
\Aff(\Z_n)=\{\s(i)=ai+b\mid a\in\+{\Z}_n,\,b\in\Z_n\}.
\end{equation}
It is a standard fact of modular arithmetic that $a\in\+{\Z}_n$ if and only if $a$ is relatively prime to $n$ \cite{Rosen}. We leave it up to the reader to show that $\Aff(\Z_n)$ can be isomorphically represented by $2\times2$ upper triangular matrices of the form $\begin{pmatrix}
a  &  b\\
0  &  1\end{pmatrix}$ with the usual matrix multiplication. 

Linear permutations $\pi_a$ will serve as convenient replacements for $\psi$ from our examples because their powers are easier to track, $\pi_a^k=\pi_{a^k}$. Every affine permutation is a product of $\pi_a$ and a power of $\k$. Indeed, $\pi_{a,b}=\k^b\pi_a$, so $\psi=\pi_{-2,3}=\k^3\pi_{-2}$. The commutation relations for $\pi_a$ are also easy to derive as in \eqref{comphikappa}, they are $\pi_a\phi=\k^{a-1}\phi\pi_a$ and $\pi_a\k=\k^a\pi_a$. The powers of $-2$ generate all units of $\Z_n$ for $n=5,7$, so we can generate all affine permutations by just multiplying powers of $\psi$ (or of $\pi_{-2}$) and powers of $\k$. Thus, the group generated by $\psi$ and $D_n$ in these dimensions is exactly  $\Aff(\Z_n)$. 

A simple consequence of Theorem \ref{pmagic criteria} is the following.
\begin{corollary}\label{AfPmag} A permutation $\pi_{a,b}\in \Aff(\Z_n)$ is dihedral if and only if $a=\pm1$, and is panmagic if and only if $a,a\pm1\in\+{\Z}_n$. It is magic if and only if it is panmagic.
\end{corollary}
\noindent The last claim follows from the fact that an affine transformation is injective if and only if it is injective at a single point. Since these properties of affine permutations entirely depend on their slopes $a$ we will also call units of $\Z_n$ dihedral or panmagic accordingly. 

It is easy to see now that the converse claim of Polya's Theorem \ref{PolyaDim} is true. Indeed, $\pi_2$ and $\pi_3$ are panmagic whenever $2,3\nmid n$ because $2,2\pm1$ and $3,3\pm1$ are units for any such $n$. And since at least one of $a-1$, $a$, and $a+1$ is divisible by $2$ and at least one by $3$ we see directly that affine panmagic permutations do not exist when $n$ is even or divisible by $3$. 

By Corollary \ref{AfPmag}, the dihedral group $D_n$ is a subgroup of $\Aff(\Z_n)$ consisting of permutations $\pi_{\e,b}$ with $\e=\pm1$. The next theorem shows that $\Aff(\Z_n)$ is the largest subgroup of $S_n$ in which $D_n$ is normal, confirming that it is the `right' group to consider.
\begin{theorem}\label{NormDn} Let $n\geq3$ and $\s\in S_n$. Then $\s D_n=D_n\s$ if and only if $\s\in \Aff(\Z_n)$.
\end{theorem}
\begin{proof} If $\s D_n=D_n\s$ then $\s\k=\pi_{\e,a}\,\s$ with $\e=\pm1$ and $a\in\Z_n$. Applying both sides to $i$, we get $\s(i+1)=\e\s(i)+a$. If $\e=-1$ then $\s(i+2)=-\s(i+1)+a=\s(i)$, so $\s$ cannot be a permutation for $n\geq3$. If $\e=1$ and we let $\s(0)=b$ then, by induction, $\s(i)=ai+b$. Since $\s$ is a permutation, $a\in\+{\Z}_n$, so $\s=\pi_{a,b}\in \Aff(\Z_n)$. 

Conversely, if $\s\in \Aff(\Z_n)$ then $\s=\pi_{a,b}=\k^b\pi_a$, and $\s D_n=D_n\s$ follows from the commutation relations for $\pi_a$.
\end{proof}
The number of units in $\Z_n$ is given by Euler's totient function $\pfi(n):=|\+{\Z}_n|$, where $|S|$ denotes the cardinality of $S$. When $n=p_1^{\a_1}\cdots p_m^{\a_m}$ is the prime factorization of $n$, we have \cite{Rosen}
\begin{equation}\label{totient}
\pfi(n)=n\left(1-\frac1{p_1}\right)\cdots\left(1-\frac1{p_m}\right).
\end{equation}
We now see from \eqref{GAmod} that $|\Aff(\Z_n)|=\pfi(n)n$. 

The number of panmagic units of $\Z_n$ can be similarly counted. By Corollary \ref{AfPmag}, $a$ is panmagic when  $a,a\pm1\in\+{\Z}_n$, so panmagic units are in $1$-$1$ correspondence with triples of consecutive units. Back in 1869, the same year when Lionnet formulated the $n$-queens problem, Victor Schemmel generalized Euler's totient to $k$-totients $\pfi_k(n)$ that count $k$-tuples of consecutive units, see \cite{Lehm}. It can be calculated analogously to  \eqref{totient}: 
\begin{equation}\label{ktotient}
\pfi_k(n)=n\left(1-\frac{k}{p_1}\right)\cdots\left(1-\frac{k}{p_m}\right)
\end{equation}
when all $p_i\geq k$, and $0$ otherwise. The number of panmagic units is thereby $\pfi_3(n)$, and we see that there are none in non-Polya dimensions. For primes, $\pfi(p)=p-1$ and $\pfi_3(p)=p-3$, as expected.

Are all panmagic permutations affine? This question received much attention in the modular $n$-queens literature, and the answer, in general, is no. They are for $n=5,7,11$, but in all Polya dimensions $n\geq13$ there exist non-affine panmagic permutations. First examples of them were constructed only in 1975 by Bruen and Dixon, and their nature and properties are still poorly understood \cite{BS2}. We will, therefore, confine our attention to affine panmagic, which already has plenty of intricacies to offer.

\section{Affine panmagic in prime dimensions}

Our identification of the Post cover group with $\Aff(\Z_n)$ for $n=5,7$ depended on $-2$ generating all units as its powers. Such units are called {\it primitive roots} in number theory and they do not exist for all $n$. Luckily, they do exist for $n=p\geq3$ prime, by a classical theorem of Gauss \cite{Rosen}. If $r$ is a primitive root then we can generate all $\pi_a$ with $\pi_r$ since $\pi_{r^k}=\pi_r^k$. And since any affine permutation is a product of $\pi_a$ and a power of $\k$ we have $\Aff(\Z_p)=\langle\pi_r,\k\rangle$ for any odd prime $p$. In this long section, we use primitive roots to generalize the algebraic properties of panmagic permutations observed for $p=5,7$ and determine their cycle structure. 

\subsection{Closure under $N$-ary multiplication}

When $n=p$ is prime, a linear permutation $\pi_a$ from Definition \ref{pia} is panmagic for any $a\neq 0,\pm1$ by Corollary \ref{AfPmag}. This means that among units only $a=\pm 1$ are not panmagic, and those are dihedral. 
\begin{corollary}[{\bf Panmagic/dihedral dichotomy}]\label{affinegeneration} When $p$ is prime, every affine permutation is either panmagic or dihedral, and has the form $\pi_r^i\k^j$, where $r$ is a primitive root modulo $p$. 
\end{corollary}
Recall that when $H$ is a normal subgroup of a group $G$ the set of cosets of $H$ with induced multiplication on them is called the quotient group and denoted $G/H$ \cite{Gallian}. We are now ready to prove the first of our main theorems about affine panmagic cosets that generalizes the $p=5,7$ examples. 
\begin{theorem}\label{primetheory} Let $p\geq5$ be prime, $r$ be its primitive root, $N:=\frac{p+1}{2}$ and $C:=\pi_rD_p$. Then for $i=1,\dots,N-2$ the cosets $C^i=\pi_r^iD_p$  consist of panmagic permutations and are closed under $N$-ary multiplication. Moreover, 
$$\langle \pi_r,D_p\rangle=\cup_{i=1}^{N-1}C^i=\Aff(\Z_p),$$
$D_p$ is its normal subgroup, and the quotient group is $\Aff(\Z_p)/D_p\simeq\Z_{N-1}=\Z_{\frac{p-1}2}$. 
\end{theorem}
\begin{proof}
The commutation relations, $\pi_r\phi=\k^{r-1}\phi\pi_r$ and $\pi_r\k=\k^r\pi_r$, imply that left and right $\pi_r$-cosets of $D_p=\langle\phi,\k\rangle$ are equal, so $C=\pi_rD_p=D_p\pi_r$. Since $D_p$ is a group we have $D_p\cdot D_p=D_p$ and, by induction, 
$$
C^i=\pi_rD_p\pi_rD_p\cdots\pi_rD_p =\pi_r^2D_p\cdots\pi_rD_p=\cdots=\pi_r^iD_p=D_p\pi_r^i\,.
$$ 
Since $r$ is a primitive root every coset of $D_p$ in $\Aff(\Z_p)$ is of the form $D_p\k^b\pi_{r^i}=D_p\pi_{r^i}=C^i$, so $\langle \pi_r,D_p\rangle=\Aff(\Z_p)$ and $D_p$ is normal in it. 

The order of $r$ is $p-1$, the size of $\+{\Z}_p$, so $r^{p-1}=1$ and $p-1$ is the smallest such power. Therefore, $r^{\frac{p-1}{2}}=-1$, as $-1$ is the only other unit that squares to $1$, and $\frac{p-1}{2}$ is the smallest power such that $r^i=\pm1$. But $\pi_r^i=\pi_{r^i}\in D_p$ if and only if $r^i=\pm1$, so the smallest $i$ such that $\pi_r^i\in D_p$ is $\frac{p-1}{2}=N-1$. Therefore, $C^{N-1}=D_p$ and all cosets are exhausted by $C^i$ with $1\leq i\leq N-1$. Moreover, $C^i\cdot C^j=\pi_r^i\pi_r^j\,D_p=\pi_r^{i+j}\,D_p=C^{i+j}$, so  $C^i\mapsto i$ is an isomorphism between the quotient group and $\Z_{N-1}$.

By Corollary \ref{affinegeneration}, every permutation in $\Aff(\Z_p)$ is either panmagic or dihedral, and all dihedral ones are in $D_p$. Since cosets partition the group we conclude that $C^i$ for $1\leq i\leq N-2$ contain only panmagic affine permutations, and (jointly) all of them. 
Finally, $C^N=CC^{N-1}=CD_p=C$, so $N$-ary products of elements of $C$ are again in $C$ and it is closed under $N$-ary multiplication.
\end{proof}
Theorem \ref{primetheory} is a global one and gives us a broad algebraic picture of how affine panmagic permutations in prime dimensions behave. But, in general, the arity $N$ it gives is not the smallest possible for every coset. Indeed, if $C^4=D_p$ then $(C^2)^2=D_p$, so $C^2$ will not just be quinternary like $C$, but also ternary. Conversely, if $C$ is $N'$-ary then we can take a product of its $N'$ elements, multiply it by $N'-1$ more elements, then by $N'-1$ more, and so on. All of those products will again be in $C$. This means that closure under $N'$-ary multiplication implies closure under $N$-ary multiplication for any $N=N'+k(N'-1)$. In particular, closure under ternary products implies closure under any odd-numbered products. 

In the light of the above ambiguity, we would like to distinguish the smallest possible arity for each coset.
\begin{definition} A subset $C$ of a group will be called strictly $N$-ary, and $N$ its strict arity, when it is closed under $N$-ary multiplication, but not under $N'$-ary multiplication with any $N'<N$. 
\end{definition}
\noindent To find strict arities, we will supplement Theorem \ref{primetheory} by a local theorem that treats each coset individually. Recall that the {\it multiplicative order} $\ord_p(x)$ of $x\in\Z_p^{\times}$ is the smallest positive integer $i$ such that $x^i\equiv1\pmod{p}$.
\begin{theorem}\label{primestrictarity} Let $p\geq5$ be prime, $a\in\Z_p^{\times}$ be panmagic, and $N:=\ord_p(a^2)+1$. Then the coset $C:=\pi_aD_p$ consists of panmagic permutations and is strictly $N$-ary. Moreover, 
$\langle\pi_a,D_p\rangle=\cup_{i=1}^{N-1}C^i$, $D_p$ is a normal subgroup in it, and the quotient group is $\langle\pi_a,D_p\rangle/D_p\simeq\Z_{N-1}$. 
\end{theorem}
\begin{proof}
The proof is analogous to the proof of Theorem \ref{primetheory} with $a$ replacing $r$, so we only point out the differences. The group $\langle\pi_a,D_p\rangle$ is no longer necessarily the whole group $\Aff(\Z_p)$, but might be its proper subgroup. The smallest $i$ such that $a^i=\pm1$ is also the smallest $i$ such that $(a^i)^2=(a^2)^i=1$, i.e. $\ord_p(a^2)$, which gives the formula for $N$. That $N$ is the strict arity follows from the fact $a^i\neq\pm1$ for $i<\ord_p(a^2)$. Indeed,  $\pi_{a^i}=\pi_a^i\not\in D_p$ for such $i$, hence $\pi_a^{i+1}\not\in C$, so $C$ is not closed under $N'$-ary multiplication with $N'=i+1<\ord_p(a^2)+1=N$.
\end{proof}
\noindent Theorem \ref{primestrictarity} allows us to find all primes $p$ that admit strictly $N$-ary panmagic cosets of $D_p$. Those must satisfy $C^{N-1}=D_p$ and not coincide with $D_p$ itself. Therefore, we need $\frac{p-1}{2}$ from Theorem \ref{primetheory} to be divisible by $N-1$, meaning that $p-1$ must be divisible by $2(N-1)$. 
\begin{corollary}\label{primeallstrictarity}
Let $p\geq5$ be prime and $N\geq3$. Then strictly $N$-ary affine panmagic cosets of $D_p$ exist if and only if $p=2(N-1)k+1$ for some integer $k$. When $p$  is of this form and $r$ is its primitive root then $\pi_{r^k}D_p$ is such a coset. In particular, strictly ternary panmagic cosets of $D_p$ exist if and only if $p=4k+1$ for some integer $k$.
\end{corollary}
\begin{proof} By Theorem \ref{primestrictarity}, strictly $N$-ary cosets of $D_p$ exist if an only if there are units $a\in\Z_p^{\times}$ with $\ord_p(a^2)=N-1$. Since all $a\in\Z_p^{\times}$ are powers of $r$, all $a^2$ are powers of $r^2$. But $\ord_p(r^2)=\frac{p-1}{2}$, so the orders of its powers, $(r^2)^k$, are exactly the divisors of $\frac{p-1}{2}$, which means that $\frac{p-1}{2}=(N-1)k$ for some integer $k$. Conversely, when this identity holds, the order of $(r^{k})^2$ will be $N-1$. Thus, $\pi_{r^k}D_p$ will be strictly $N$-ary by Theorem \ref{primestrictarity}.
\end{proof}
\noindent Primes of the form $4k+1$, like 5, 13, 17, 29, 37, 41..., are famous in number theory. They were singled out already by Albert Girard and Pierre de Fermat, who stated that they are exactly the primes representable as the sums of two perfect squares. The first proof was given by Euler. They also happen to be the only primes for which $-1$ is a quadratic residue, the square of another unit \cite{Rosen}. Unsurprisingly, $4k+1$ primes appear in connection with the modular queens problem as well \cite{BS1}, and they will reappear in Corollary \ref{4cyc}. 

\subsection{Panmagic cycle types}

We will now turn from algebra to the structure of individual affine panmagic permutations. Recall that intrinsic character of a permutation is described by its {\it cycle type}, the number and lengths of cycles in its disjoint cycle decomposition \cite{Gallian}. One can easily check that all panmagic permutations in dimension $5$ have one fixed point and one $4$-cycle, and they also all belong to a single coset of $D_5$. However, in the $\pi_2D_7$ coset, we encounter two cycle types represented by $\pi_2=(1\,2\,4)(3\,6\,5)(7)$ and $\pi_5=(1\,5\,4\,6\,2\,3)(7)$. Further computations confirm that there are always at most two cycle types in each panmagic coset of $D_p$, and all cycles in each, other than the fixed point, have the same length. Moreover, when there are two cycle types the cycle length in one of them is twice that in the other.

A first step towards proving these observations is to recall that conjugate permutations have the same cycle types (essentially, because they permute numbers in identical ways up to relabeling) \cite{Gallian}. Then we observe that every affine permutation $\pi_{a,b}$ with $a\neq1$ is conjugate to $\pi_a$. Indeed, $\s=\k^{-t}\pi_a\k^t$ with $t\equiv-(a-1)^{-1}b\pmod{p}$. Thus, we only need to consider cycle types of $\pi_a$, which is much less daunting since their powers are easily tracked. This already confirms one of the above observations. As the  slopes of permutations in $\pi_aD_p$ are only $\pm a$, there can be at most two cycle types in it, that of $\pi_a$ and that of $\pi_{-a}$. 
\begin{corollary}\label{primecycletypes}
Let $a\in\Z_p^{\times}$ be panmagic. Then the disjoint cycle decomposition of $\pi_a$ consists of one fixed point and $\frac{p-1}{\ord_p(a)}$ cycles of equal length $\ord_p(a)$. Moreover, every permutation in $\pi_aD_p$ has the cycle type of $\pi_a$ or of $\pi_{-a}$.
\end{corollary}
\begin{proof} 
Recall that we identify residue $0$ with $p$ for affine permutations. Clearly, $\pi_a(0)=0$, and this gives us the fixed point $p$. The rest of residues are units $u$ and $\pi_a(u)=au$, so the cycle that starts with $u$ is of the form 
$$
(u\ au\ \cdots\ a^{m-1}u),
$$
where $m$ is the smallest positive integer with $a^mu=u$. Since $u$ is a unit this is equivalent to $a^m=1$, so $m=\ord_p(a)$ for all $u$ and is their common cycle length. Since there are $p-1$ units in $\Z_p^{\times}$ there are $\frac{p-1}{m}$ such cycles.
\end{proof}
Thus, there are at most two cycle types in each panmagic coset. When is there just one? We already know that this happens exactly when $\pi_a$ and $\pi_{-a}$ have the same type. By Corollary \ref{primecycletypes}, a necessary and sufficient condition is that $\ord_p(-a)=\ord_p(a)$. To make it more explicit, we will relate the  multiplicative orders of $a$ and $a^2$ (for any $n$).
\begin{lemma}\label{ord+-a}
If $\ord_n(a^2)$ is even then $\ord_n(a)=\ord_n(-a)=2\ord_n(a^2)$. If $\ord_n(a^2)$ is odd then one of $\ord_n(-a)$ and $\ord_n(a)$ is equal to $\ord_n(a^2)$ and the other to $2\ord_n(a^2)$. 
\end{lemma}
\begin{proof} 
By a standard result of group theory \cite{Gallian}, 
\begin{equation}\label{orda^2} 
\ord_n(a^2)=\begin{cases}\ord_n(a),\,\ord_n(a)\text{ odd}\\\frac12\ord_n(a),\,\ord_n(a)\text{ even.}\end{cases}
\end{equation}
Hence, if $\ord_n(a^2)$ is even then so is $\ord_n(a)$. Indeed, if it were odd then by \eqref{orda^2} $\ord_n( a)=\ord_n(a^2)$ would also be odd, a contradiction. But then, by \eqref{orda^2} again, $\ord_n( a)=2\ord_n(a^2)$, and the same argument applies to $-a$.

Now suppose that $\ord_n(a^2)$ is odd. Since  $\ord_n(a^2)=\ord_n\big((-a)^2\big)$, formula \eqref{orda^2} implies that $\ord_n(\pm a)$ is either $\ord_n(a^2)$ or $2\ord_n(a^2)$. For definiteness, suppose $\ord_n(a)=\ord_n(a^2)$. Since it is odd $(-a)^{\ord_n(a^2)}=- a^{\ord_n(a)}=-1$, so $\ord_n(-a)\neq\ord_n(a^2)$ and it must be $2\ord_n(a^2)$.
\end{proof}
Our next result now follows easily.
\begin{theorem}\label{onecycle} Let $p\geq5$ be prime, $a\in\Z_p^{\times}$ be panmagic, and $N:=\ord_p(a^2)+1$. If $N$ is odd then all permutations in $\pi_aD_p$ have the same cycle type with the cycle lengths, aside from the fixed point, all equal to $2(N-1)$. If $N$ is even then permutations in $\pi_aD_p$ have two cycle types with the cycle lengths, aside from the fixed point, all equal to $N-1$ in one of them and $2(N-1)$ in the other.
\end{theorem}
\begin{proof} If $N$ is odd then $\ord_p(a^2)=N-1$ is even and Lemma \ref{ord+-a} implies that $\ord_p(-a)=\ord_p(a)=2\ord_p(a^2)=2(N-1)$. If  $N$ is even then $\ord_p(a^2)=N-1$ is odd. By Lemma \ref{ord+-a}, one of $\ord_p(-a)$ and $\ord_p(a)$ is then equal to $\ord_p(a^2)=N-1$ and the other is twice that.
\end{proof}
\noindent As a simple observation, transpositions, i.e. $2$-cycles, can only appear in cosets when $N=2$, i.e. when the coset is $D_p$ itself. In other words, affine panmagic permutations have no transpositions in prime dimensions. 

Theorem \ref{onecycle} means that cosets of $\pi_aD_p$ have permutations with the same cycle type if and only if their strict arity $N$ is odd. But then $\ord_p(a)=2\ord_p(a^2)=2(N-1)$ is divisible by $4$. Since the order of any element must divide the order of $\Z_p^{\times}$ by Lagrange's theorem \cite{Gallian}, we conclude that $4\mid p-1$ and $p$ is of the form $4k+1$. But this is the exact same condition as in Corollary \ref{primeallstrictarity}\,! It holds if and only if strictly ternary panmagic cosets of $D_p$ exist. In turn, Theorem \ref{onecycle} implies that all permutations in such cosets have the same cycle type. They have only $4$-cycles and a single fixed point. Summarizing, we proved the following neat equivalence that highlights the special panmagic of $4k+1$ primes.
\begin{corollary}\label{4cyc}
The following conditions on a prime $p$ are equivalent: 

\textup{(i)} $D_p$ has panmagic cosets with all permutations of the same cycle type; 

\textup{(ii)} $D_p$ has panmagic cosets of strictly odd arity; 

\textup{(iii)} $D_p$ has strictly ternary panmagic cosets;

\textup{(iv)} $p=4k+1$ for a positive integer $k$.
\end{corollary}

\section{$N$-ary subgroups}\label{NarySub}

Before concluding, we wish to put things into a broader algebraic perspective and tie up some loose ends. Theorems \ref{primetheory} and \ref{primestrictarity} provide a template for how sets closed under $N$-ary multiplication arise as cosets, but it is not yet clear how much of it is specific to those examples and how much is essential. Moreover, closure under multiplication does not quite entitle us to call them $N$-ary {\it groups}. If the binary case is our guide, we also need an analog of closure under taking inverses. We will briefly develop a bit of $N$-ary group theory to fill in those gaps.

$N$-ary groups can be defined abstractly, like ordinary groups, but to keep our exposition elementary we will only define $N$-ary {\it sub}groups of ordinary groups. Those are the only ones we need, and one can prove that any abstract $N$-ary group is isomorphic to an $N$-ary subgroup \cite{GBV,Shah}.  
\begin{definition}\label{NaryG}
A subset $C\subseteq G$ of a group $G$ is called its $N$-ary subgroup when for any $N$ elements $a_1,\dots,a_N\in C$, the product $a_1\cdots a_N\in C$; and for any $a\in C$, the element $a^{3-2N}\in C$. 
\end{definition} 
\noindent The first condition replaces closure under the binary multiplication by an $N$-ary one and the second replaces closure under inversion. When $N=2$ we recover the definition of the ordinary (binary) subgroup. 

In the ternary case, $a^{3-2N}=a^{-3}$, but we still get closure under taking inverses because $a^{-1}=aaa^{-3}\in C$ as a triple product. In fact, for $N\geq3$ we can generally replace the condition $a^{3-2N}\in C$ with a simpler equivalent one, $a^{2-N}\in C$. Indeed, $a^{2-N}=a^{N-1}a^{3-2N}$ and $a^{3-2N}=a^{N-3}\left(a^{2-N}\right)^3$, and both are $N$-ary products for $N\geq3$. The element $a^{2-N}$ is called {\it skew to $a$} and is typically used as the replacement for the inverse in $N$-ary group theory \cite{GBV,Shah}. 

We did not say anything about the identity element $e$, and that is by design. In the binary case, $e$ will be in $C$ because we can multiply any element by its inverse. But also conversely, if $e\in C$ then $C$ is a binary subgroup because we can reproduce binary products by taking $N$-ary ones with $N-2$ copies of $e$ in them. Indeed, under Definition \ref{NaryG} any binary subgroup is also $N$-ary for any $N\geq3$. 

Nonetheless, there are non-binary $N$-ary subgroups for $N>2$. In fact, many of them are quite familiar, albeit rarely thought about this way. For example, odd integers form a ternary (additive) subgroup of $\Z$. Indeed, while the sum of two odd integers is not odd, the sum of any three is. The inverse (negative) of an odd integer is also odd. Similarly, odd permutations form a non-binary ternary subgroup of $S_n$. Non-zero purely imaginary numbers $i\+{\R}$ form a (multiplicative) ternary subgroup of non-zero complex numbers $\+{\C}$. Indeed, $ia\cdot ib\cdot ic=i(-abc)$ and $(ia)^{-1}=i(-a)$.

One can check that $\psi D_5=\pi_3D_5$, the panmagic ternary coset from Section \ref{57PMag}, does contain the inverses of all of its elements and is a ternary subgroup. The quaternary coset $\psi D_7=\pi_5D_7$ does not; in fact, its inverses belong to the other ternary panmagic coset $\pi_3D_7$. But it does contain the skews $a^{-2}$ of its elements and is a quaternary subgroup. Fortunately, we do not need to go back and check that our cosets closed under $N$-ary multiplication were also closed under skewing. This is because all of them were finite, and, as with ordinary finite subgroups, closure under multiplication is enough to be an $N$-ary subgroup.
\begin{theorem} Suppose $C\subseteq G$ is closed under $N$-ary multiplication and finite. Then it is an $N$-ary subgroup of $G$.
\end{theorem}
\begin{proof} Let $C^{N-1}$ denote the set of $(N-1)$-element products of elements of $C$. We claim that it is a binary subgroup. Indeed, a product of two $(N-1)$-products of elements of $C$ has the total of $2(N-1)=N+(N-2)$ factors. The first $N$ multiply to an element of $C$ by assumption, so it is again a product of $N-1$ factors from $C$. If $C$ is finite then so is $C^{N-1}$. Since it is closed under binary multiplication it must be a subgroup by a standard result of group theory \cite{Gallian}.

Let $a\in C$, then $a^{N-1}\in C^{N-1}$. Since a binary subgroup contains inverses of its elements, $a^{1-N}=(a^{N-1})^{-1}\in C^{N-1}$, i.e. is a product of $N-1$ elements of $C$. But then $a^{2-N}=aa^{1-N}$ is a product of $N$ elements of $C$ and must be in $C$ by assumption. Thus, $C$ is an $N$-ary subgroup.
\end{proof}
It turns out that the trick for obtaining $N$-ary subgroups from cosets, which we exploited in Theorems \ref{primetheory} and \ref{primestrictarity}, is quite general. It constitutes the `easy' direction of the Post coset theorem \cite{GBV,Shah}.
\begin{theorem}[Post]\label{PostConv} Let $H\subset\H\subseteq G$ be a pair of (binary) subgroups of $G$ with $H$ normal in $\H$ and $\H/H\simeq\Z_{N-1}$. Then every coset of $H$ in $\H$ is an $N$-ary subgroup of $G$.
\end{theorem}
\begin{proof} Let $C=bH$ for $b\in\H$. Since $H$ is normal, for any $h\in H$, we have $hb=bh'$ for some $h'\in H$, and since $\H/H$ is cyclic of order $N-1$, we have $b^{N-1}\in H$ for any $b\in\H$. Therefore, for $a_i=bh_i$ with $h_i\in H$, by induction,
$$
a_1\cdots a_N=bh_1 bh_2\cdots bh_N=b(bh_1'h_2b\cdots bh_N)=\dots=b(b^{N-1}\hat{h})\in bH=C.
$$
Similarly, since $(b^{-1})^{N-1}\in H$, moving $b^{-1}$-s to the left we obtain for $a=bh$, 
$$
a^{2-N}=(h^{-1}b^{-1})^{N-2}=(b^{-1})^{N-2}\hat{h}=b\big((b^{-1})^{N-1}\hat{h}\big)\in bH=C.
$$
Thus, $C$ is an $N$-ary subgroup.
\end{proof}
\noindent In Theorem \ref{primetheory}, $G=S_p$ was the symmetric group, $H=D_n$ was the dihedral group, and $\H=\Aff(\Z_p)$ was the general affine group. One can check that for the purely imaginary numbers we have $G=\+{\C}$, $H=\+{\R}$, and $\H=\langle i,\+{\R}\rangle=\+{\R}\cup i\+{\R}$.

The hard direction of the Post coset theorem is that any $N$-ary group arises as a Post coset up to isomorphism, i.e. it is isomorphic to a coset of an index $N-1$ normal subgroup of a larger group with the cyclic quotient. It was proved by Post in his 1940 seminal paper on $N$-ary groups. In general, given an $N$-ary subgroup $C\subseteq G$, the group $H=C^{N-1}$ is called its {\it Post associate} and the group $\H=\langle C\rangle$ its {\it Post cover}. In our panmagic examples, the Post associate was always $D_p$, while the Post cover could be the entire group 
$\Aff(\Z_p)$ or its subgroup of the form $\langle\pi_a,D_p\rangle$ with a panmagic unit $a$. 

Note that $N$ does not, in general, have to be the strict arity of $C$, and that the Post cover may depend on which $N$ it is considered under. It follows from the Post coset theorem that the strict arity is $\ord(C)+1$, where $\ord(C)$ is the order of $C$ as an element of the quotient group $\H/H$. We will not elaborate further and refer the interested reader to \cite{GBV,Shah} and references therein. 

In conclusion, let us neatly repackage Theorems \ref{primetheory} and \ref{primestrictarity} in terms of $N$-ary group theory.
\begin{theorem} Let $p\geq5$ be prime. If $a\in\Z_p^{\times}$ is panmagic then $\pi_aD_p$ is a strictly $N$-ary panmagic subgroup of $S_p$ with $N=\ord_p(a^2)+1$, the Post associate $D_p$ and the Post cover $\langle\pi_a,D_p\rangle$. Moreover, all $\pi_aD_p$ with $a\neq0,\pm1$ are $N$-ary panmagic subgroups of $S_p$ with  common $N=\frac{p+1}{2}$, common Post associate $D_p$, and common Post cover $\Aff(\Z_p)$.
\end{theorem}
\noindent The appearance of $\Aff(\Z_p)$ as the common Post cover is a consequence of the existence of primitive roots in prime dimensions. Indeed, when $r$ is a primitive root, permutations of a single coset $\pi_rD_p$ generate the entire group of affine panmagic permutations. In turn, $D_p$ is the common Post associate because it is the largest subgroup that preserves panmagic in those dimensions. These observations point to what should change and what should stay the same in non-prime dimensions.

\section{Conclusions and open problems}

We developed an algebraic  theory of affine panmagic permutations that focuses on their algebraic and number-theoretic properties. Their matrices are the simplest kind of panmagic squares and they also represent configurations of non-attacking queens on a toroidal chessboard. But it is treating them as permutations that unlocked many new results. We hope that our algebraic approach and its generalizations will be useful for solving other problems. Let us briefly outline some adjacent topics and open problems that lie ahead for those willing to pursue the subject further.

We expect that most of our results on $N$-ary groups and cycle types of panmagic permutations generalize to composite dimensions. However, one would need a different approach to proving them because the only composites satisfying the Polya condition that have primitive roots are odd prime powers. Even for them, the panmagic/dihedral dichotomy fails -- there exist units that are neither panmagic nor dihedral. Moreover, in composite dimensions, $D_n$ is no longer the largest subgroup of $S_n$ whose elements preserve panmagic, and one needs to consider Post cosets of its extensions. The cycle types of panmagic permutations also become more complex, and no longer have all cycles of the same length aside from the fixed point as in Theorem \ref{onecycle}.

%Theorem \ref{GAcycletypes} describes cycle types not only of panmagic permutations but of all affine permutations with fixed points. However, already dihedral permutations may have no fixed points. The cyclic shift $\k=\pi_{1,1}$ is one of them and it has a single cycle of length $n$. Cycles of affine permutations are known as linear congruential sequences in random number generation and they are studied in \cite{Mars}, but cycle types of multidihedral and lifted dihedral permutations without fixed points do not seem to be known in general. A partial result in this direction, that $\pi_{a,b}$ has a single cycle of length $n$ when $a\equiv1\pmod{p}$ for every $p\mid n$ and $b$ is relatively prime to $n$, is known as the Hull-Dobell theorem \cite{HD,Mars}. In our terminology, $a$ from the Hull-Dobell theorem is none other than a lift of unity modulo $n$, a special case of lifted dihedral units. 

Affine panmagic permutations seem to be closely related to the uniform step method for constructing ``natural" panmagic squares, those filled with natural numbers from $1$ to $n^2$. Planck used them to construct modular queens solutions already in 1900 \cite{BS2}. The method only works in Polya dimensions, and there are intriguing parallels between Lehmer's results on it \cite{Lehm} and ours. More recently, constructions linking natural panmagic squares to non-affine panmagic permutations were also discovered \cite{BS1}.

Counting all panmagic permutations in dimension $n$ is known to be a hard problem, and even good estimates are hard to come by. It is known that there are more than $2^{\sqrt{\frac{p-1}2}}$ of them for prime $n=p$, and, conjecturally, the count is asymptotically $\sim n^{\a n}$ for some $\a>0$ \cite{BS2}. Since there are only $\pfi_3(n)n$ affine panmagic permutations and $\pfi_3(n)\leq n$, the vast majority of panmagic permutations are non-affine for large $n$. While some special constructions of such permutations are known \cite{BS2}, they are far from producing the entire set or revealing its algebraic structure. 

Non-affine panmagic permutations present new, more complex questions. For example, are there $N$-ary groups consisting of them? Non-affine panmagic cosets of $D_n$ are ruled out by Theorem \ref{NormDn}, but there may be other groups that work. One would have to find pairs of subgroups of $S_n$, one normal in the other, with non-affine panmagic cosets. Some candidate groups that act on general panmagic permutations are considered in \cite{Eng}.

The linear span of permutation matrices of any $N$-ary panmagic subgroup is an $N$-ary algebra of panmagic squares. Many such examples are provided by our Corollary \ref{primeallstrictarity}. So far, only ternary panmagic algebras have been considered in the literature and they have an appealing alternative description in terms of linear algebra. Consider an invertible matrix $Q$ that commutes with both $P_{\phi}$ and $P_{\k}$, and let $E$ be the square matrix with all entries $1$. Matrices $A$ such that $AQ+QA$ is a scalar multiple of $E$ are called $Q$-regular, and they are panmagic squares that form a ternary algebra \cite{Nord}. Thompson's algebra, the linear span of $\pi_2D_5$, is $Q$-regular with $Q=\frac12I+P_{\k}+P_{\k}^{-1}$. Can such $Q$ be found for other ternary groups from our Corollary \ref{primeallstrictarity}? And conversely, which $Q$-regular algebras are spanned by ternary groups and how can those groups be recovered from $Q$? More generally, what is a linear algebra description of $N$-ary panmagic algebras, and when are they linear spans of $N$-ary panmagic groups?

Although affine panmagic permutations are a small fraction of all of them, this tells us little about their share of the space of panmagic squares. Indeed, there are $n!$ permutations in $S_n$, but the dimension of the linear span of all permutation matrices (which is the space of semimagic squares) is only $(n-1)^2+1$. This means that there are massively many linear relations among permutations matrices, and it is {\it a priori} possible that panmagic permutation matrices (or already affine ones) span the entire space of panmagic squares for Polya $n$. Do they? 

When we consider only squares with positive entries and positive linear combinations, the answer is no for $n>5$, as proved in \cite{AlKi}. For odd $n$, the space of panmagic squares is known to be $(n-2)^2$-dimensional \cite{Hou}, so we could answer the question if we counted linearly independent panmagic  permutation matrices. Questions about linear relations among matrices of group elements belong to the group representation theory. For the affine case, this suggests looking into representation theory of $\Aff(\Z_n)$ and its subgroups. But this is a task for another day.

\end{document}